\documentclass[11pt,a4paper]{amsart}
\usepackage{amssymb,amsmath}
\usepackage{graphicx}
\usepackage{subcaption}
\usepackage{comment}
\usepackage{tikz}
\usetikzlibrary{arrows}

\def\({\left(}
\def\){\right)}
\def\Nx{\nabla_x}
\def\Cal{\mathcal}

\def\eb{\varepsilon}

\newcommand{\be}{\begin{equation} }
\newcommand{\ee}{\end{equation} }

\def \and{\qquad\text{and}\qquad}

\def\Bbb{\mathbb}
\def\Dt{\partial_t}
\def\Dx{\Delta_x}

\def\({\left(}
\def\){\right)}
\def\Nx{\nabla_x}
\def\divv{\operatorname{div}}
\def\eb{\varepsilon}
\def\Cal{\mathcal}
\def\eb{\varepsilon}

\def\R {\mathbb R}

\def\<{\left<}
\def\>{\right>}

\def \and{\qquad\text{and}\qquad}

\def\Bbb{\mathbb}
\def\Dt{\partial_t}
\def\Dx{\Delta_x}
\def\R {\mathbb R}

\newtheorem{proposition}{Proposition}[section]
\newtheorem{theorem}[proposition]{Theorem}
\newtheorem{corollary}[proposition]{Corollary}
\newtheorem{lemma}[proposition]{Lemma}
\theoremstyle{definition}
\newtheorem{definition}[proposition]{Definition}
\newtheorem{remark}[proposition]{Remark}
\newtheorem{example}[proposition]{Example}
\newtheorem{problem}[proposition]{Problem}
\numberwithin{equation}{section}
\def\be{\begin{equation}}
\def\ee{\end{equation}}
\def\bp{\begin{proof}}
\def\ep{\end{proof}}
\def \au {\rm}

\def \bk {\it}
\def \no#1#2#3 {{\bf #1} (#3), #2.}
\def \eds#1#2#3 {#1, #2, #3.}

\title[Determining functionals for dissipative PDEs]
{Determining functionals and finite-dimensional reduction for dissipative PDEs revisited}
\author[V. Kalantarov, A. Kostianko,  and  S. Zelik]
{Varga Kalantarov${}^0$,  Anna Kostianko${}^{1,2}$, and Sergey Zelik${}^{1,3,4}$}

\address{${}^0$ Department of mathematics,
\newline\indent Ko{\c c} University, Rumelifeneri Yolu, Sariyer, Istanbul, Turkey}

\address{${}^1$ \phantom{e}School of Mathematics and Statistics, Lanzhou University, Lanzhou\\ 730000,
P.R. China}
\address{${}^2$  Imperial College, London SW7 2AZ, United Kingdom.}

\address{${}^3$ University of Surrey, Department of Mathematics, Guildford, GU2 7XH, United Kingdom.}

\address{${}^4$ Keldysh Institute of Applied Mathematics, Moscow, Russia}

\email{a.kostianko@imperial.ac.uk}
\email{s.zelik@surrey.ac.uk}

\begin{document}

\begin{abstract}  { We study } the properties of linear and non-linear determining functionals for dissipative dynamical systems generated by PDEs. The main attention is payed to the lower bounds for the number of such functionals. In contradiction to the common paradigm, it is shown that the optimal number of determining functionals (the so-called determining dimension) is strongly related to the proper dimension of the set of equilibria of the considered dynamical system rather than to the dimensions of the global attractors and the complexity of the dynamics on it. In particular, in the generic case  where the set of equilibria is finite, the determining dimension equals to one (in complete agreement with the Takens delayed embedding theorem) no matter how complex the underlying dynamics is. The obtained results are illustrated by a number of explicit examples.
\end{abstract}

\thanks{ This work is partially supported by the grant 19-71-30004 of RSF (Russia),
  and the Leverhulme grant No. RPG-2021-072 (United Kingdom) }
\maketitle
\tableofcontents

\section{Introduction}\label{s0}
It is believed that in many cases the limit dynamics generated by dissipative PDEs is essentially finite-dimensional and can be effectively described by finitely many parameters (the so-called order parameters in the terminology of I. Prigogine) governed by a system of ODEs describing its evolution in time (which is usually referred as an Inertial Form (IF) of the system considered), see \cite{BV92,CV02,ha,MZ08,R01,SY02,T97} and references therein. A mathematically rigorous interpretation of this conjecture naturally leads to the concept of an Inertial Manifold (IM). By definition, IM is an invariant smooth (at least Lipschitz) finite-dimensional  submanifold in the phase space  which is exponentially stable and possesses the so-called exponential tracking property, see \cite{CFNT89,FST88,M-PS88,M91,R94,Z14} and references therein. However, the existence of such an object requires rather restrictive extra assumptions on the considered system, therefore, an IM may a priori not exist for many interesting equations arising in applications, see \cite{EKZ13,KZ18,KZ17,Z14} for more details. In particular, the existence or non-existence of an IM for 2D Navier-Stokes remains an open problem and for the simplified models of 1D coupled Burgers type systems such a manifold may not exist, see \cite{KZ18,KZ17,Z14}.
\par
By this reason, a number of weaker interpretations of the finite-di\-men\-sio\-na\-lity conjecture are intensively studied during the last 50 years.
One of them is related to the concept of a global attractor and interprets its finite-dimensionality in terms of Hausdorff or/and box-counting dimensions keeping in mind   the Man\'e projection theorem. Indeed, it is well-known that under some relatively weak assumptions
 the dissipative system considered   possesses a compact global attractor of  finite box-counting dimension, see \cite{BV92,MZ08,R01,T97} and references therein. { Then the}  corresponding IF can be constructed   by projecting the considered dynamical system to a generic plane of sufficiently large, but finite dimension.  The key drawback of this approach is related to the fact that the obtained IF is only H\"older continuous (which is not enough in general even for the uniqueness of solutions of the reduced system of ODEs). Moreover, as recent examples show, the limit dynamics may remain in a sense infinite-dimensional despite the existence of a H\"older continuous IF (at least, it may demonstrate some features, like super-exponentially attracting limit cycles, traveling waves in Fourier space, etc., which are impossible in classical dynamics), see \cite{Z14} for more details.
\par
An alternative even more weaker approach (which actually has been historically the first, see \cite{FP67,Lad72}) is related to the concept of {\it determining functionals}. By definition, a system of (usually linear) continuous functionals $\Cal F:=\{F_1,\cdots,F_N\}$ on the phase space $H$ is called (asymptotically) determining if for any two trajectories $u_1(t)$ and $u_2(t)$, the convergence
$$
F_n(u_1(t))-F_n(u_2(t))\to0 \ \text{ as $t\to\infty$ for all $n=1,\cdots,N$}
$$
 implies that $u_1(t)\to u_2(t)$ in $H$ as $t\to\infty$.   Thus, in the case where $\Cal F$ exists, the limit behavior of a trajectory $u(t)$ as $t\to\infty$ is uniquely determined by the behaviour of finitely many scalar quantities $F_1(u(t))$, $\cdots$, $F_N(u(t))$, see e.g. \cite{Chu} for more details.
\par
To be more precise, it has been shown in \cite{FP67,Lad72} that, for the case of  Navier-Stokes equations in 2D, a system generated by  first $N$ Fourier modes is asymptotically determining if $N$ is large enough.
Later on the notions of determining nodes, volume elements, etc., have been introduced and various upper bounds for the number $N$ of elements in such systems have been obtained for various   dissipative PDEs (see, e.g., \cite{Chu1,FMRT, FT, FTT,JT} and references therein). More general classes of determining functionals have been introduced in \cite{Cock1,Cock2}, see also \cite{Chu,Chu1}.
\par
We emphasize  from the very beginning that the existence of the finite system $\Cal F$ of determining functionals
{\it does not} imply in general that the quantities $F_n(u(t))$, $n\in1,\cdots,N$ obey
 a system of ODEs, therefore, this approach does not a priori justify   the finite-dimensionality conjecture. Moreover, these quantities usually obey only the delay differential equations (DDEs) whose phase space remain infinite-dimensional, see Section \ref{s2} below. Nevertheless, they may be useful for many purposes, for instance, for establishing the controllability of an initially infinite dimensional system by finitely many modes (see e.g. \cite{AT14}), verifying the uniqueness of an invariant measure for random/stochasitc PDEs (see e.g. \cite{kuksin}), etc.  We also mention more recent but very promising applications of determining functionals to data assimilation problems where the values of functionals $F_n(u(t))$ are interpreted as the results of observations and the theory of determining functionals allows us to build new methods of restoring the trajectory $u(t)$ by the results of observations, see \cite{AT14,AT13,OT08,OT03} and references therein.
\par
It {  is } worth to note here that, despite the long history and big number of research papers and interesting results, the determining functionals are essentially less understood in comparison with IMs and global attractors. Indeed, the general strategy which most  
 of the papers in the field follow on is to fix some very specific class of functionals (like Fourier modes or the values of $u$ in some prescribed points in  space (the so-called nodes), etc.) and to give as accurate as possible {\it upper} bounds for the number of such functionals which is sufficient to generate a determining system. These upper bounds are then expressed in terms of physical parameters of the considered system, e.g. in terms of the so-called Grashof's number when 2D Navier-Stokes equations are considered, see \cite{Chu,OT08} for details. These estimates often give the bounds for the number $N$ which are compatible to the box-counting dimension of the global attractor.
However,
the questions of sharpness of these upper bounds and of   finding the appropriate lower bounds remain completely unstudied. The only natural lower bound for $N$ which is available is related to the dimension of the set of equilibria for the considered system.
\par
 An exception is the case of one spatial variable, where the determining systems of very few number of functionals (which is compatible with the dimension of the equilibria set) are known, see e.g. \cite{kukavica}. However, these 1D results somehow support the Takens conjecture (see \cite{Tak})  that the number of determining functionals should be very small (a generic single functional should be determining for a generic system whose set of equilibria is finite by the Sard theorem) and clearly contradict the conjecture that the above mentioned estimates (which are compatible with the dimension of the attractor) are sharp. We also mention that the restriction to the class of linear determining functionals looks artificial especially when a non-linear dissipative system is considered.
\par
The main aim of the present paper is to shed some light on the above mentioned questions. For simplicity, we restrict ourselves by considering only the following semilinear abstract parabolic equation in a Hilbert space~$H$:
\begin{equation}\label{0.abs}
\Dt u+Au-f(u)=g,
\end{equation}
where $A:D(A)\to H$ is a linear self-adjoint positive operator with compact inverse, $f$ is a given nonlinearity which is  globally Lipschitz and subordinated to  the linear part $A$ of the equation and $g$ is a given external force, see Section \ref{s2} for precise conditions.
\par
One of the main results of the manuscript is the following theorem,
\begin{theorem}\label{Th0.main} Let the operator $A$ and the nonlinearity $f$ satisfy some natural assumptions stated in Section \ref{s2} and let the right-hand side $g\in H$ be chosen in such a way that the set $\Cal R$ of equilibria points of equation \eqref{0.abs} is finite. Then, there is a prevalent set of continuous maps $F: H\to\R$ each of them can be chosen as an asymptotically determining functional for problem \eqref{0.abs}. Moreover, such a functional can be chosen from the class of polynomial maps of sufficiently   { large} order.
\end{theorem}
The proof of this theorem is strongly based on the H\"older continuous version of the Takens delay embedding theorem and is given in Section \ref{s3}.
\par
The next corollary describes the dynamical properties of the scalar quantity $Z(t):=F(u(t))$.
\begin{corollary}\label{Cor0.delay} Let the assumptions of Theorem \ref{Th0.main} hold and let $F:H\to\R$ be a smooth determining functional constructed there. Then, there is a sufficiently small delay $\tau>0$, a sufficiently large $k\in\Bbb N$ and a continuous function $\bar\Theta:\R^k\to\R$ such that the quantity $Z(t):=F(u(t))$   obeys the following scalar ODE with delay:
\begin{equation}\label{0.delay}
\frac d{dt}Z(t)=\bar \Theta(Z(t-\tau),\cdots,Z(t-k\tau)),\ \ t\in\R
\end{equation}
for every $u(t)$ belonging to the global attractor $\Cal A$.
\end{corollary}
We now relax the assumption that the set $\Cal R$ of equilibria is finite. To this end, we introduce the so-called embedding dimension $\dim_{emb}(\Cal R,H)$ as the minimal number $M$ such that there is an injective continuous map $\Phi:\Cal R\to\R^M$. Obviously, the number of functionals in any determining system cannot be less than $\dim_{emb}(\Cal R,H)$:
\begin{equation}
\dim_{det}(S(t),H)\ge\dim_{emb}(\Cal R,H),
\end{equation}
 where $\dim_{det}(S(t),H)$ is the minimal number $N$ such that equation \eqref{0.abs} possesses a determining system of $N$ continuous (not necessarily linear) functionals. The next proposition is the analogue of Theorem \ref{Th0.main} for {  the} general case.
 \begin{proposition}\label{Prop0.R} Let all of the assumptions of Theorem \ref{Th0.main} except { of}  the finiteness of $\Cal R$ be satisfied. Then, the determining dimension of the dynamical system $S(t)$ generated by equation \eqref{0.abs} satisfies:
 \begin{equation}\label{0.est}
 \dim_{emb}(\Cal R,H)\le\dim_{det}(S(t),H)\le \dim_{emb}(\Cal R)+1.
 \end{equation}
 \end{proposition}
 Thus, the number of functionals in the "optimal" determining system  {  depends}  only   {  on}  the embedding properties of the equilibria set $\Cal R$ to Euclidean spaces and is {\it not related}
  {  to} the size of the global attractor or the complexity of the dynamics on it.
Moreover, as shown in Example \ref{Ex4.wave} below, it may be unrelated even   {  to} the dissipativity of the system considered.
 \par
 We also note that the above mentioned "paradox" with the case of one spatial dimension can be naturally resolved in light of Theorem \ref{Th0.main} and Proposition \ref{Prop0.R}. Indeed, although the optimal number of determining functionals is always related
{  to} the set of equilibria $\Cal R$ only, in 1D case the situation is essentially simpler since the equilibria  usually solve the system of ODEs and, therefore, the embedding dimension of $\Cal R$ cannot be greater than the order of this system of ODEs (of course, if this system is smooth enough for the uniqueness theorem to  hold). Thus, in  1D case we always have a good estimate for $\dim_{emb}(\Cal R,H)$ which is independent of the physical parameters of the system, see Examples \ref{Ex4.dir}, \ref{Ex4.per} and \ref{Ex4.inf} where the explicit form of possible determining functionals are given for various cases of semilinear heat equations.
 \par
 In contrast to this, in the multidimensional case, the elements of $\Cal R$ are usually the solutions of elliptic PDEs and we do not have convenient estimates for the size of $\Cal R$, we only know that, for generic external forces $g$, $\Cal R$ is finite due to the Sard theorem, but for exceptional choices of $g$, we do not have anything more than the obvious estimate
 \begin{equation}
 \dim_{emb}(\Cal R,H)\le 2\dim_B(\Cal R,H)+1\le \dim_B(\Cal A,H)+1
 \end{equation}
 and the box-counting dimension of $\Cal R$ may indeed be compatible with the appropriate dimension of the global attractor $\Cal A$, see Example \ref{Ex4.deg}. Thus, the assumption that the external forces $g$ should be generic   {  appears to be} unavoidable if we want to get sharp results. Note also, that as shown in Example \ref{Ex4.linbad}, the class of linear functionals is not sufficient to get the above mentioned results, so considering polynomial functionals is also unavoidable.
\par
The paper is organized as follows. Some preliminary results on a general theory of determining functionals are collected in Section \ref{s1}. In particular, a bit stronger concept of functionals separating trajectories on the attractor is introduced there. The non-equivalence of determining and separating functionals is shown in Example \ref{Ex4.sep}.
\par
The classical theory of determining Fourier modes for equation \eqref{0.abs} including the related Lyapunov-Schmidt reduction and recent applications to restoring the trajectory $u(t)$ by the observation data related to determining modes are discussed in Section \ref{s2}.
\par
The main results of the paper are stated and proved in Section \ref{s3}.
\par
Section \ref{s4}, which is also one of the central sections of the paper, is devoted to various examples and counter-examples related  
{  to} the theory of determining functionals. Finally, some open problems which are interesting from our point of view are presented in Section \ref{s5}.

\section{Preliminaries}\label{s1}
In this section we recall the basic facts about the determining functionals, introduce the notations and state the necessary definitions, see e.g. \cite{Chu,OT08} for more detailed exposition. Let $H$ be a Banach space and $S(t): H\to H$, $t\ge0$ be a semigroup acting on it, i.e.
\begin{equation}\label{1.sem}
S(t+l)=S(t)\circ S(l) \ \text{ for all $t,l\ge0$  and }\ S(0)=\operatorname{Id}.
\end{equation}
This semigroup will be referred as a dynamical system (DS) in the phase space $H$ and its orbits $u(t)=S(t)u_0$, $u_0\in H$, $t\ge0$ will be treated as (semi)trajectories of this dynamical system.
\par
We are now ready to give the definition of determining functionals which will be used throughout the paper.
\begin{definition}\label{Def1.as-det} A system $\Cal F:=\{F_1,\cdots, F_N\}$ of possibly non-linear continuous functionals $F_i:H\to\R$ is called asymptotically determining for the DS $S(t)$ if, for any two trajectories $u(t)$ and $v(t)$ of this DS,
\begin{equation}\label{1.F-van}
\lim_{t\to\infty}(F_i(u(t))-F_i(v(t))=0,\ \ i=1,\cdots, N
\end{equation}
implies that
\begin{equation}\label{1.van}
\lim_{t\to\infty}\|u(t)-v(t)\|_H=0.
\end{equation}
\end{definition}
Let us assume in addition that the considered DS is {\it dissipative}, i.e., there exist  positive constants $C$ and $\alpha$ and a monotone increasing function $Q$ such that
\begin{equation}\label{1.dis}
\|u(t)\|_H\le Q(\|u_0\|_H)e^{-\alpha t}+C
\end{equation}
for every trajectory $u(t)$,
 and possesses  the so-called {\it global attractor} $\Cal A$ in the phase space $H$. We recall that the set $\Cal A$ is a global attractor of the DS $S(t)$ if the following assumptions are satisfied:
 \par
 1) $\Cal A$ is a compact set in $H$;
 \par
 2) $\Cal A$ is strictly invariant: $S(t)\Cal A=\Cal A$ for all $t\ge0$;
 \par
 3) $\Cal A$ attracts the images of all bounded sets of $H$ as $t\to\infty$, i.e., for any bounded $B\subset H$ and any neighbourhood $\Cal O(\Cal A)$ of the attractor, there exists $T=T(B,\Cal O)$ such that
 $$
 S(t)B\subset \Cal O(\Cal A) \ \ \text{if $t\ge T $},
 $$
 see \cite{BV92,T97} for more details.
\par
In this case, we may give an alternative definition which is often  a bit simpler to verify in applications.

\begin{definition}\label{Def1.atr-det} Let $S(t)$ be a DS in $H$ which possesses a global attractor $\Cal A$ in it.  Then, a system $\Cal F:=\{F_1,\cdots,F_N\}$ of continuous functionals is called separating on the attractor $\Cal A$ if for any two complete trajectories $u(t)$ and $v(t)$, $t\in\R$ belonging to the attractor, the identity
\begin{equation}\label{1.det-attr}
F_i(u(t))=F_i(v(t)),\ \ \text{for all $t\in\R$ and $i=1,\cdots, N$}
\end{equation}
implies that $u(t)\equiv v(t)$, $t\in\R$.
\end{definition}
The next standard proposition shows that under some natural assumptions, separating system on the attractor is automatically asymptotically determining.
\begin{proposition}\label{Prop1.det} Let $S(t)$ be a dissipative DS which possesses a global attractor $\Cal A$ in the phase space $H$. Assume also that all its trajectories $u(t)$ are continuous in time (belong to the space $C_{loc}(\R_+,H)$) and the map $S: u_0\to S(t)u_0$ is continuous as the map from $H$ to $C_{loc}(\R_+,H)$. Then any separating on the attractor system $\Cal F$ is asymptotically determining.
\end{proposition}
\begin{proof}   {  Assume} that $\Cal F$ is not asymptotically determining. Then, there exist two trajectories $u(t)$ and $v(t)$ such that $F_i(u(t))\to F_i(v(t))$ as $t\to \infty$, but $u(t)$ does not tend to $v(t)$ as $t\to\infty$. Therefore, there exists a sequence $t_n\to\infty$ such that $\|u(t_n)-v(t_n)\|_H\ge\eb>0$.
\par
Recall that, by assumptions, $S(t)$ in $H$ possesses a global attractor $\Cal A$. Let us consider the set $K_+\subset C_{loc}(\R_+,H)$ of all semi-trajectories of the initial DS. Then the semigroup of shifts
$$
(T(h)u)(t):=u(t+h),\  t,h\ge0
$$
acts on $K_+$ (i.e. $T(h)K_+\subset K_+$). This semigroup is often referred as the trajectory dynamical system associated with the initial DS $S(t)$, see \cite{CV02, MZ08} for more details. Moreover, according to our assumptions, this semigroup is homeomorphically conjugated to the initial DS $S(t)$. In particular, the orbits $T(h)u$, $h\in\R_+$ are asymptotically compact in $C_{loc}(\R_+,H)$ for any $u\in K_+$. Thus, without loss of generality, we may assume that
$$
T(t_n)u\to\bar u,\ \ \ T(t_n)v\to\bar v
$$
in $C_{loc}(\R_+,H)$, where $\bar u$ and $\bar v$ are two trajectories belonging to the attractor. Since we assume that $F_i(u(t))\to F_i(v(t))$ as $t\to\infty$, from the continuity of $F_i$ we conclude that
$$
F_i(\bar u(t))\equiv F_i(\bar v(t)),  \ \forall t\in\R.
$$
On the other hand, since $\|u(t_n)-v(t_n)\|_H\ge\eb>0$, we conclude that $\bar u(0)\ne\bar v(0)$. The last statement contradicts the assumption that $\Cal F$ is separating on the attractor and finishes the proof of the proposition.
\end{proof}
\begin{remark}\label{Rem1.strange} As we will show in Example \ref{Ex4.sep}  below, the asymptotically determining system may be not separating on the attractor, so the above two definitions are not equivalent. However, the assumptions of Definition \ref{Def1.atr-det} are usually easier to verify, so we mainly deal in what follows with separating systems on the attractor.
\par
 The analogue of Proposition \ref{Prop1.det} holds for some non-dissipative systems as well.
Indeed, if we replace the existence of a global attractor by the assumption that the $\omega$-limit set
$\omega(u_0)$ of any point $u_0\in H$ is not empty and compact and require \eqref{1.det-attr} to be satisfied
for any complete bounded trajectory $u(t)$, $t\in\R$, the assertion of the proposition will remain true.
We also note that the continuity assumptions on the semigroup can be relaxed  {  to}  the standard requirements that
the maps $S(t)$ are continuous for every fixed $t$.
\end{remark}
Since we are interested in the number $N$ of the "optimal" system of determining functionals, it is natural to give the following definition.
\begin{definition}\label{Def3.det-dim} Let $S(t)$ be a DS acting in a Banach space $H$. A determining dimension $\dim_{det}(S(t),H)$ is the minimal number $N\in\Bbb N$ such that there exists asymptotically determining system $\Cal F$ which consists of $N$ functionals.
\end{definition}
Our primary goal is to find or estimate the determining dimension for a given DS $S(t)$. We start with some obvious, but useful observations. Namely, let $\Cal R\subset H$ be the set of all equilibria of the DS $S(t)$ and let $\Cal F=\{F_1,\cdots,F_N\}$ be an asymptotically determining system of functionals which generates a continuous map $F:H\to\R^N$  via $F(u):=(F_1(u),\cdots, F_N(u))$. Obviously, this map must be {\it injective} on $\Cal R$. This gives a natural lower bound for the determining dimension:
\begin{equation}\label{1.lower}
\dim_{det}(S(t),H)\ge \dim_{emb}(\Cal R),
\end{equation}
where the embedding dimension of a set $\Cal R\subset H$ is a minimal $N\in\Bbb N$ such that there exists a continuous embedding of $\Cal R$ to $\R^N$.
\par
Let us now discuss straightforward upper bounds. To this end, we assume that the DS considered is dissipative and possesses a global attractor $\Cal A$ of finite box-counting (fractal) dimension
\begin{equation}\label{1.dim-ass}
\dim_B(\Cal A,H)<\infty.
\end{equation}
This assumption is true for many interesting classes of dissipative systems generated by PDEs, see \cite{BV92,MZ08,T97} and references therein. Then, by Man\'e projection theorem, see e.g. \cite{3,hunt,R11}, a generic projector $P:H\to V_N$ on a plane $V_N$ of dimension $N\ge2\dim_B(\Cal A,H)+1$ is one-to-one on the attractor $\Cal A$. Let us write this projector in the form
\begin{equation}\label{1.mp}
Ph=\sum_{n=1}^N (F_n,h)v_n,\ \ h\in H,
\end{equation}
where $F_n\in H^*$ and $\{v_n\}_{n=1}^N\in H$ is the base in $V_N$. Then the system of linear functionals $\{F_n(h)=(F_n,h)\}_{n=1}^N$ is obviously separating on the attractor $\Cal A$. This gives the desired estimate:
\begin{equation}
\dim_{det}(S(t),H)\le 2\dim_B(\Cal A,H)+1.
\end{equation}
Thus, we have proved the following proposition.
\begin{proposition}\label{Prop1.rough} Let the DS $S(t)$ be dissipative and possess a global attractor $\Cal A$ of finite fractal dimension. Then,
\begin{equation}\label{1.r}
\dim_{emb}(\Cal R)\le \dim_{det}(S(t),H)\le 2\dim_B(\Cal A,H)+1,
\end{equation}
where $\Cal R\subset H$ is the set of all equilibria of the DS considered.
\end{proposition}
\begin{remark}\label{Rem1.fut} As we will see below, the upper bound in \eqref{1.r} is too rough and can be replaced by the corresponding dimension of the equilibria set $\Cal R$. Note also that the determining dimension may be defined not as a minimum dimension with respect to all continuous functionals, but for functionals satisfying some extra restrictions, for instance, for {\it linear} functions only (or even linear functionals of some special form, e.g., determining nodes, etc.). This may potentially lead to new non-trivial results. However, the restriction for determining functionals to be linear does not look natural especially when the non-linear DS is considered. Moreover, as shown in Example \ref{Ex4.linbad}, linear functionals may fail to distinguish different periodic orbits  with the same period (even in linear systems) which leads to unnecessary increase of the number of functionals in a determining system. As we also see below, the class of polynomial functionals is {  rich} enough to overcome such problems.
\end{remark}

\section{Determining modes: a classical example}\label{s2}
In this section, we recall an old result on the determining modes for a semilinear parabolic equation  which we are planning to use in the sequel, see e.g. \cite{Chu} or \cite{HR} for a more detailed exposition. Namely, we consider the following abstract functional model:
\begin{equation}\label{2.abs}
\Dt u+Au-f(u)=g,\ \ u\big|_{t=0}=u_0,
\end{equation}
where $H$ is a Hilbert space, $A=A^*>0$ is a positive self-adjoint unbounded operator in $A$ with a compact inverse, $g\in H^{-\alpha}$ and $f$ is a given non-linear map which is subordinated to the operator $H$ and is bounded and globally Lipschitz:
\begin{equation}\label{2.lip}
\begin{cases}
1.\  \|f(u)\|_{H^{-\alpha}}\le -C,\ \forall u\in H^\alpha;\\
2.\   \|f(u_1)-f(u_2)\|_{H^{-\alpha}}\le L\|u_1-u_2\|_{H^\alpha},\ \forall u_1,u_2\in H^\alpha,
\end{cases}
\end{equation}
where $\alpha\in[0,1)$ and $H^\alpha:=D(A^{\alpha/2})$ and $C>0$. It is well-known that under the above assumptions,  equation \eqref{2.abs} generates a dissipative DS in the space $H$ which possesses a global attractor $\Cal A$. Moreover, this attractor has the finite box-counting dimension:
\begin{equation}
\dim_B(\Cal A,H)<\infty,
\end{equation}
see e.g. \cite{hen,R01} for more details. Note that \eqref{2.abs} is a functional model for many dissipative PDEs arising in applications including reaction-diffusion equations, Navier-Stokes system, etc. The extra assumption that $f$ is {\it globally} Lipschitz is not a big restriction if the existence of an absorbing ball is established and can be achieved by the standard cut off procedure outside of this ball.
\par
Let $\{e_n\}_{n=1}^\infty$ be an orthonormal base of eigenvectors of the operator $A$ with the corresponding eigenvalues $\{\lambda_n\}_{n=1}^\infty$ and let $P_N:H\to H$ be the orthoprojector to the first $N$ eigenvectors of $A$. Denote also $Q_N:=1-P_N$. Consider the
 system $\Cal F=\{F_1,\cdots,F_N\}$ of linear functionals generated by the corresponding Fourier modes $F_n(u):=(u,e_n)$. The next proposition tells us that the system $\Cal F$ is determining if $N$ is large enough.

 \begin{proposition}\label{Prop2.modes} Let the above assumptions hold and let $N$ satisfy the inequality
\begin{equation}\label{2.one}
 L<\lambda_{N+1}^{1-\alpha}.
 \end{equation}
 Then the system $\Cal F$ of first $N$ Fourier modes is asymptotically determining for the DS generated by equation \eqref{2.abs}.
 \end{proposition}
  \begin{proof} Indeed, let $u_1(t)$ and $u_2(t)$ be two trajectories belonging to the attractor. Such that $P_Nu_1(t)=P_Nu_2(t)$ for all $t\in\R$. Let $v(t):=u_1(t)-u_2(t)$. Then this function solves
  \begin{equation}\label{2.dif}
  \Dt v(t)+Av(t)=[f(u_1(t))-f(u_2(t))],\ \ t\in\R.
  \end{equation}
  Taking an inner product of equation \eqref{2.dif} with $v(t)$ and using the Lipschitz continuity of the map $f$ and the fact that $P_Nv(t)\equiv0$, we arrive at
  $$
  \frac12\frac d{dt}\|v(t)\|^2_H+(\lambda_{N+1}^{1-\alpha}-L)\|v(t)\|_{H^\alpha}^2\le0
  $$
  and, therefore,
  $$
  \|v(t)\|^2_H\le e^{-\beta(t-s)}\|v(s)\|^2_H,\ \ s\le t
  $$
  for some positive $\beta$. Since $u_1$ and $u_2$ belong to the attractor, $\|v(s)\|_H$ remains bounded as $s\to-\infty$ and, passing to the limit $s\to-\infty$, we get $v(t)\equiv0$. Thus, $\Cal F$ is separating on the attractor and the proposition is proved.
  \end{proof}
\begin{remark}\label{Rem2.NS} Similar result has been initially proved for 2D Navier-Stokes problem, see \cite{FP67,Lad72}, and has been extended later for many other classes of dissipative PDEs. For simplicity we consider here only the case of abstract parabolic equations although the result remains true for much more general types of equations, including e.g. damped wave equations, etc, see \cite{Chu}. The main drawback of the construction given above is that it gives {\it one-sided} estimate for the number $N$ of functionals in the determining system $\Cal F$  and, as we will see below, there are no reasons to expect that this estimate is sharp. Moreover, the Takens conjecture as well as the results stated in the next section allow us to guess that in a "generic" situation we have $N=1$.
\end{remark}
We now reformulate the result of the proposition in terms of Lyapunov-Schmidt reduction. For simplicity, we restrict ourselves to consider the case $\alpha=0$ only although everything remains true (with minor changes in a general case $\alpha\in[0,1)$ as well). Let us  rewrite equation \eqref{2.abs} in terms of lower $u_+(t):=P_Nu(t)$ and higher $u_-(t):=Q_Nu(t)$ Fourier modes, namely,
\begin{equation}\label{2.sys}
\begin{cases}
\frac d{dt}u_+(t)+Au_+(t)-P_Nf(u_+(t)+u_-(t))=g_+,\\
\frac d{dt}u_-(t)+Au_-(t)-Q_Nf(u_+(t)+u_-(t))=g_-.
\end{cases}
\end{equation}
The next lemma shows that the higher modes part  $u_-(t)$ is uniquely determined if its lower modes part $u_+(t)$ are known.
\begin{lemma}\label{Lem2.Lyap} Let \eqref{2.one} be satisfied and let $\alpha=0$. Then, for every $u_+\in C_b(\R_-,H)$ there is a unique solution $u_-\in C_b(\R_-,H)$ of the second equation of \eqref{2.sys}. Moreover, the solution operator $\Phi:C_b(\R_-,H)\to C_b(\R_-,H)$ is Lipschitz continuous in the following sense:
\begin{equation}\label{2.llip}
\|\Phi(u_+^1)-\Phi(u_+^2)\|_{C_{e^{\beta t}}(\R_-,H)}\le K\|u_+^1-u_+^2\|_{C_{e^{\beta t}}(\R_-,H)}
\end{equation}
for some positive $\beta$ and $K$ which are independent of $u_+^1,u_+^2\in C_b(\R_-,H)$.
\end{lemma}
\begin{proof} For a given $u_+\in C_b(\R_-,H)$ let us consider the following problem
\begin{equation}\label{2.app}
\Dt u_-+Au_--Q_Nf(u_-+u_+)=g,\ \ u_-\big|_{t=-M}=0,
\end{equation}
where $M>0$ is fixed.
Obviously, this problem has a unique solution $u_{-,M}\in C([-M,0],H)$. Moreover, multiplying equation \eqref{2.app} scalarly in $H$ by $u_-$, we get
\begin{multline}
\frac12\frac d{dt}\|u_-\|^2_H+\lambda_{N+1}\|u_-\|^2_H\le\\\le (f(u_-+u_+)-f(0),u_-)+(f(0)+g,u_-)\le \\\le L\|u_-\|^2_H+L\|u_+\|_H\|u_-\|_H+(\|f(0)\|_H+\|g\|_H)\|u_-\|_H.
\end{multline}
Using \eqref{2.one}, we arrive at
$$
\frac d{dt}\|u_-(t)\|^2_H+\beta\|u_-(t)\|^2_H\le C(1+\|g\|^2_H+\|u_+(t)\|^2_H)
$$
for some positive constants $\beta$ and $C$. Integrating this inequality, we see that
$$
\|u_-(t)\|_{H}^2\le C_1(1+\|g\|^2_H)+C_1\int_{-M}^te^{-\beta(t-s)}\|u_+(s)\|^2_H\,ds,\ \ t\ge -M.
$$
This estimate in particular shows that the functions $u_{-,M}(t)$ are uniformly bounded with respect to $M$ (since $u_+\in C_b(\R_-,H)$). We claim that the sequence $\{u_{-,M}\}_{M=1}^\infty$ is  Cauchy in $C_{loc}(\R_-,H)$. Indeed, let $M_1>M_2$ and $v_{M_1,M_2}(t):=u_{-,M_1}(t)-u_{-,M_2}(t)$. Then this function solves the equation
\begin{multline}
\Dt v_{M_1,M_2}+Av_{M_1,M_2}=Q_N[f(u_{-,M_1}+u_+)-f(u_{-,M_2}+u_+)],\\ v_{M_1,M_2}\big|_{t=-M_2}=u_{-,M_1}(-M_2).
\end{multline}
Multiplying this equation by $v_{M_1,M_2}$ and arguing as in the proof of Proposition \ref{Prop2.modes} using that $u_{-,M_1}(-M_2)$ is uniformly bounded with respect to $M_1$ and $M_2$, we infer
$$
\|v_{M_1,M_2}(t)\|_H^2\le Ce^{-\beta(t+M_2)},\ \ t\ge-M_2,
$$
 so $u_{-,M}(t)$ is indeed a Cauchy sequence. Passing now to the limit $M\to\infty$, we get the desired solution $u_-\in C_b(\R_-,H)$ of the second equation of \eqref{2.sys}. Thus, it only remains to verify the uniqueness and estimate \eqref{2.llip}. To this end, we take two solutions $u_-^1(t)$ and $u_-^2(t)$ which correspond to different functions $u_+^1(t)$ and $u_+^2(t)$ and take their difference $v(t)=u_-^1(t)-u_-^2(t)$. Writing out the equation for $v(t)$ and arguing as in the proof of Proposition \ref{Prop2.modes}, we end up with the following inequality:
 \begin{equation}\label{2.dif-est}
 \frac d{dt}\|v(t)\|^2_H+\beta\|v(t)\|^2_H\le C\|u_+^1(t)-u_+^2(t)\|_H^2.
 \end{equation}
 Integrating this inequality, we get
 \begin{multline}
 \|v(t)\|^2_H\le \|u_+^1(-M)-u_+^2(-M)\|_H^2e^{-\beta(t+M)}+\\+
 C_\beta\int_{-M}^te^{-\beta(t-s)}\|u_+^1(s)-u_+^2(s)\|_H^2\,ds.
\end{multline}
Using now that all functions involved are bounded as $t\to-\infty$, we may pass to the limit $M\to\infty$ in the last inequality and get
\begin{multline}
 \|v(t)\|^2_H\le C_\beta\int_{-\infty}^te^{-\beta(t-s)}\|u_+^1(s)-u_+^2(s)\|_H^2\,ds\le \\\le C'_\beta\sup_{s\in\R_-}\left\{e^{-\beta(t-s)/2}\|u_+^1(s)-u_+^2(s)\|_H^2\right\}
\end{multline}
which gives the desired estimate \eqref{2.llip} and finishes the proof of the lemma.
\end{proof}
Thus, according to Lemma \ref{Lem2.Lyap}, the value of $u_-(t)$ can be found by the map $\Phi$ if the trajectory $u_+(s)$, $s\le t$ is known. Namely
$$
u_-(t)=\Phi((u_+)_t)(0):=\Phi_0((u_+)_t),
$$
where the function $(u_+)_t\in C_b(\R_-,H)$ is given by $(u_+)_t(s):=u_+(t+s)$. This allows us to reduce the initial equation \eqref{2.abs} at least on the attractor to the following {\it delayed} system of $N$ ODEs:
\begin{equation}\label{2.dode}
\frac d{dt} u_+(t)+Au_+(t)-P_Nf(u_+(t)+\Phi_0((u_+)_t))=g.
\end{equation}
\begin{remark}\label{Rem2.nd} We see that determining modes (at least in a dissipative system \eqref{2.abs}) are responsible for the reduction of the initial PDE to a system of  ODEs with {\it delay} (DDEs) where the number of equations in the reduced system is equal to $N$. Note that the phase  space for the obtained system of DDEs remains {\it infinite-dimensional} (since we have an infinite delay in \eqref{2.dode}, it is natural to take the space $C_b(\R_-,H)$ with the topology of $C_{loc}(\R_-,H)$ as the phase space of this problem, see \cite{ha} and references therein). Thus, determining modes {\it do not produce} any finite-dimensional reduction and only reduce the initial PDE to the system of DDEs whose dynamics is still a priori infinite-dimensional. In particular, the number $N$ or the determining dimension {\it is not a priori} related to the effective number of degrees of freedom of the dissipative system considered. We also mention that more advanced technique related to {\it inertial manifolds} allows us to express $u_-(t)$ through $u_+(t)$ as a local in time function without  delay
$$
u_-(t)=\Phi_0(u_+(t)),\ \ \Phi_0\,:\, P_NH\to Q_NH
$$
(under further rather restrictive assumptions on $A$ and $f$, see \cite{BV92} for more details). In this case,  equation \eqref{2.dode} become a system of ODEs without delay and the finite-dimensional reduction holds.
\end{remark}
We conclude this section by one more related result which allows us to restore the trajectory $u(t)$ in "real time" if the values of its determining modes are known (say, from observations) and which may be useful for data assimilation, see \cite{AT13,OT08} and references therein.

\begin{proposition}\label{Prop2.assym} Let $u(t)$ be a given trajectory of the dynamical system generated by equation \eqref{2.abs} and let condition \eqref{2.one} be satisfied. Let the constant $K$ be large enough and let $v(t)$ solve the following problem:
\begin{equation}\label{2.as}
\frac d{dt}v(t)+Av(t)-f(v(t))+K(P_N v(t)-P_N u(t))=g.
\end{equation}
Then, the following estimate holds:
\begin{equation}\label{2.conv}
\|v(t))-u(t)\|_H\le \|v(0)-u(0)\|_He^{-\beta t}, \ \ t\ge0
\end{equation}
for some positive constant $\beta$ which is independent of $u$ and $v$.
\end{proposition}
\begin{proof} Indeed, let $w(t):=v(t)-u(t)$. Then this function solves
$$
\Dt w(t)+Aw(t)-[f(v(t))-f(u(t))]+K P_Nw(t)=0.
$$
Taking a scalar product of this equation with $w(t)$ in $H$ and using the Lipschitz continuity of $f$, we get
$$
\frac12\frac d{dt}\|w(t)\|^2_H+\|w(t)\|_{H^1}^2-L\|w(t)\|^2_{H^\alpha}+K\|P_Nw(t)\|^2_H\le0.
$$
Moreover, arguing as in the proof of Proposition \ref{Prop2.modes}, we get
\begin{multline}
\|w\|_{H^1}^2-L\|w\|^2_{H^\alpha}+K\|P_Nw\|^2_H=(\|Q_Nw\|_{H^1}^2-L\|Q_Nw\|^2_{H^\alpha})+\\+
(\|P_Nw\|_{H^1}^2-L\|P_Nw\|_{H^\alpha}^2+K\|P_Nw\|^2_H)\ge\\\ge
(\lambda_{N+1}^{1-\alpha}-L)\|Q_Nw\|^2_{H^\alpha}+(K-L\lambda_N^\alpha)\|P_Nw\|_{H}^2\ge \beta\|w\|^2_H
\end{multline}
for some positive $\beta$ if $K>L\lambda^\alpha$. This gives the estimate
$$
\frac12\frac d{dt}\|w(t)\|^2_H+\beta\|w(t)\|^2_H\le0
$$
and finishes the proof of the proposition.
\end{proof}
\begin{remark} As we will see below, the results which are similar (although slightly more complicated) to Lemma \ref{Lem2.Lyap} and Proposition \ref{Prop2.assym} hold for a general system of determining functionals.
\end{remark}

\section{The Takens delay embedding theorem and determining functionals}\label{s3}
In this section, we prove that at least for the case of equations \eqref{2.abs} and generic external forces $g$, the determining dimension of the corresponding DS is one. We start with recalling two preliminary results.
\begin{proposition}\label{Prop3.gen} Let the operator $A$ and the nonlinearity $f$ satisfy the assumptions of Section \ref{s2}.
 Then, there is a dense  in $H^{-\alpha}$ set of external forces $g$ for which the corresponding set $\Cal R_g$ of equilibria is finite.
\end{proposition}
Indeed, this is a standard corollary of the Sard lemma, see e.g. \cite{BV92} for more details.
\begin{proposition}\label{Prop3.per} Let the operator $A$ and the nonlinearity $f$ satisfy the assumptions of Section \ref{s2}. Then there exists $T_0>0$ depending on $A$ and $f$ only such that the period $T$ of any non-trivial periodic orbit $u(t)$ of problem \eqref{2.abs} is not smaller than $T_0$ (i.e. $T\ge T_0$).
\end{proposition}
For the proof of this fact, see \cite{R11}.
\par
We are now ready to state a version of the Takens delayed embedding theorem for the DS $S(t)$ generated by problem \eqref{2.abs} which is our main technical tool in this section.
\begin{theorem}\label{Th3.Tkn} Let $\Cal A$ be the global attractor of equation \eqref{2.abs} and let the natural number
$$
k\ge(2+\dim_B(\Cal A))\dim_B(\Cal A)+1
$$
be fixed. Assume also that the set $\Cal R$ of equilibria is finite and $\tau>0$ is fixed in such a way that $k\tau\ge T_0$ (where $T_0$ is the same as in Proposition \ref{Prop3.per}). Then there exists a dense (actually a prevalent) set of Lipschitz maps $F:H\to\R$ such that the map
\begin{equation}\label{3.tak}
F(k,u):=(F(u),F(S(\tau)u),\cdots,F(S((k-1)\tau)u))
\end{equation}
is one-to-one on the attractor $\Cal A$. Moreover, the map $F$ can be chosen from the class of polynomials of degree $2k$.
\end{theorem}
The proof of this theorem is given in \cite{R11}.
\begin{remark} A   { bit non-standard} condition on $k$ is related   {  to}
the fact that the H\"older Man\'e projection theorem is essentially used in the proof and the maximal H\"older exponent there is
restricted by the so-called dual thickness exponent of the attractor, see \cite{R11}. In many cases related   {  to}
semilinear parabolic equations, this thickness exponent is known to be zero and then we return to
 the standard condition $k\ge 2\dim_B(\Cal A)+1$ in the Takens theorem.
\end{remark}
\begin{corollary}\label{Cor3.det} Under the assumptions of Theorem \ref{Th3.Tkn},  functional $F$ constructed in this theorem is asymptotically determining and, therefore,
\begin{equation}
\dim_{det}(S(t),H)=1.
\end{equation}
\end{corollary}
\begin{proof} Indeed, let $u_1(t)$ and $u_2(t)$ be two complete trajectories belonging to the attractor and let $F(u_1(t))\equiv F(u_2(t))$ for all $t\in\R$. Then, since the map \eqref{3.tak} is one-to-one on the attractor, we conclude that $u_1(t)\equiv u_2(t)$ and $F$ is separating on the attractor. This finishes the proof of the corollary due to Proposition \ref{Prop1.det}.
\end{proof}
Let us take a functional $F$ constructed in Theorem \ref{Th3.Tkn} and let $\bar{\Cal A}:=F(k,\Cal A)\subset\R^{k}$. Then, since $\Cal A$ is compact, $F(k,\cdot)$ is a homeomorphism between $\Cal A$ and $\bar{\Cal A}$. Let us introduce one more map
\begin{equation}\label{3.theta}
\Theta: \bar{\Cal A}\to H,\ \  \Theta(u):=S(k\tau)F(k,u)^{-1}.
\end{equation}
By the construction, this map is continuous and, for every trajectory $u(t)$, $t\in\R$, belonging to the attractor, we have the identity
\begin{equation}\label{3.delay1}
u(t)=\Theta(F(u(t-k\tau),\cdots,F(u(t-\tau))),\ \ t\in\R.
\end{equation}
In particular, denoting $Z(t):=F(u(t))\in\R$, we get
\begin{equation}\label{2.Z}
Z(t)=F(\Theta(Z(t-k\tau),\cdots, Z(t-\tau)))
\end{equation}
which gives the reduction of the dissipative dynamics generated by \eqref{2.abs} on the attractor to a {\it scalar}  equation with delay.
 However, equation \eqref{2.Z} looks not friendly (e.g., does not possess any smoothing/compactness  property on a finite time interval), so it
 {  seems} better to replace it by delay differential equation in the spirit of \eqref{2.dode}.
To this end, we assume in addition that $\alpha=0$ and  our functional $F$ is smooth and use the relation \ref{3.delay1}
in order to find the derivative $\Dt u$ through equation \eqref{2.abs}.
Using in addition that the operator $A$ is continuous on the attractor as the map from $H$ to $H$ (see e.g. \cite{R11}), we end up with
\begin{multline}
\Dt u(t)=-A\Theta(F(u(t-k\tau),\cdots,F(u(t-\tau)))+\\+f( \Theta(F(u(t-k\tau),\cdots,F(u(t-\tau))))+g
\end{multline}
and therefore
\begin{multline}\label{3.Z-ode}
\frac d{dt}Z(t)=\bar \Theta(Z(t-k\tau),\cdots,Z(t-\tau))), \ \text{where}\\
\bar\Theta(\xi_1,\cdots,\xi_{k}):=\\=F'(\Theta(\xi_1,\cdots,\xi_{k}))[-A\Theta(\xi_1,\cdots,\xi_{k})
+f(\Theta(\xi_1,\cdots,\xi_{k}))+g]
\end{multline}
Thus, we have proved the following result which is somehow analogous to Lemma \ref{Lem2.Lyap}.
\begin{corollary} Let the operator $A$ and the non-linearity $f$ satisfy the assumptions of Section \ref{s2} with $\alpha=0$ and let the set of equilibria $\Cal R$ be finite. Then, there exists a continuous function $\bar \Theta:\R^k\to\R$ (where $k$ is the same as in Theorem \ref{Th3.Tkn}) such that the dynamics of \eqref{2.abs} on the attractor is homeomorphically embedded to the dynamics of scalar DDE \eqref{3.Z-ode}.
\end{corollary}
Indeed, the natural phase space of problem \eqref{3.Z-ode} is $C[-k\tau,0]$ and,  by the construction,
the map $\Theta_Z:u_0\to F(S(\cdot+k\tau)u_0)$ gives a homeomorphic embedding of the attractor $\Cal A$ to $C[-k\tau,0]$. Moreover, the dynamics of \eqref{2.abs} on $\Cal A$ is conjugated to the dynamics of \eqref{3.Z-ode} (also by the construction). It only remains to note that the function $\bar\Theta$ is initially defined on $\bar A$ only, but it can be extended to the whole $\R^k$, e.g. by Tietze theorem.
\par
We now discuss the analogue of Proposition \ref{Prop2.assym}. To this end, we fix $N\in\Bbb N$ the same as in Proposition \ref{Prop2.assym} and introduce the finite-dimensional map
\begin{equation}
\Theta_N(\xi_1,\cdots,\xi_k):=P_N\Theta(\xi_1,\cdots,\xi_k),\   \Theta_N : \bar{\Cal A}\to\R^N
\end{equation}
which can be extended by Tietze theorem to a continuous map from $\R^k$ to $\R^N$. Then the following result holds.
\begin{proposition}\label{Prop3.assym} Let the assumptions of Corollary \ref{Cor3.det} holds. Let also the trajectory $u(t)\in\Cal A$ be given and $F:H\to\R$ be an asymptotically determining functional. Consider the problem
\begin{multline}\label{3.as}
\Dt v(t)+Av(t)-f(v(t))+\\+K(P_Nv(t)-\Theta_N(F(u(t-k\tau)),\cdots,F(u(t-\tau))))=0,
\end{multline}
where $K$ is large enough. Then, estimate \eqref{2.conv} holds.
\end{proposition}
Indeed, by the construction,
$$
P_Nu(t)=\Theta_N(F(u(t-k\tau),\cdots,F(u(t-\tau))))
$$
and the assertion follows immediately from Proposition \ref{Prop2.assym}.
\begin{remark}As we have already mentioned, the assumption $\alpha=0$ is introduced for simplicity only. All of the results stated below remain true (with very minor changes) in a general case $\alpha\in[0,1)$ as well.
\end{remark}
To conclude the section, we briefly discuss the case where the extra generic assumption about the finiteness of $\Cal R$ is not posed. This may be useful e.g. for the case when problem \eqref{2.abs} possesses some physically relevant symmetries which are forbidden to break. The analogue of Corollary \ref{Cor3.det} now reads.
\begin{corollary}\label{Cor3.det-R} Let the operator $A$ and the nonlinearity $f$ satisfy the assumptions of Section \ref{s2}. Then the determining dimension of the DS associated with equation \eqref{2.abs} possesses the following estimate:
\begin{equation}\label{3.goodest}
\dim_{emb}(\Cal R)\le\dim_{det}(S(t),H)\le\dim_{emb}(S(t),H)+1,
\end{equation}
where $\Cal R$ is the corresponding set of equilibria.
\end{corollary}
\begin{proof}[Sketch of the proof]  The proof of this statement consists of a revision of the proof given in \cite[Theorem 14.7]{R11} adapting it to our case. First, the reduction to finite-dimensional case works exactly as in \cite{R11}, so we only need to revise the proof of the finite-dimensional version, see \cite[Theorem 14.5]{R11}. Since in our case all periodic orbits of small period are equilibria, three cases considered in the proof of this theorem are reduced to two cases.
\par
{\it Case 1:} at least one of two points $x,y\in\Cal A$ does not belong to $\Cal R$. In this case, a single functional $F$ which distinguishes such points can be constructed exactly in the same way as in the proof of   \cite[Theorem 14.5]{R11} (estimating the rank of the corresponding matrix and using \cite[Lemma 14.4]{R11}).
\par
{\it Case 2:} $x,y\in\Cal R$. To distinguish such points we use the fact that there exists a homeomorphic embedding
$\phi:\Cal R\to\R^{\dim_{emb}(\Cal R)}$, so we just add the coordinates of $\phi$ to our system of functionals and get a system $\{F,\phi_1,\cdots,\phi_{\dim_{emb}(\Cal R)}\}$ (we implicitly assume that $\phi$ is already extended in a continuous way to the whole $H$). This gives the desired upper bound for the determining dimension. The lower bound follows from \eqref{1.lower} and the corollary is proved.
\end{proof}
\begin{remark} Recall that the embedding dimension of $\Cal R$ possesses the following estimates:
\begin{equation}\label{3.cont}
\dim_{emb}(\R)\le 2\dim_{top}(\R)+1\le2\dim_B(\R)+1\le2\dim_B(\Cal A)+1<\infty,
\end{equation}
where $\dim_{top}$ is a topological (Lebesgue covering) dimension, see e.g. \cite{R11}. Mention also that, if we control the embedding dimension by the topological one, we may also claim that the determining functionals are dense and if we control it by box-counting dimension, we may get the system of determining functionals all of them except of one are linear and this one is polynomial. In order to avoid the technicalities, we leave the rigorous proof of this fact to the reader.
\end{remark}

\section{Examples}\label{s4}
In this section, we illustrate the theory by a number of  examples where the determining functionals can be explicitly found. Some of them are known, but the rest  {  are new to the best of our knowledge}. We start with the example which shows the difference between asymptotically determining and separating on the attractor functionals.

\begin{example}\label{Ex4.sep} Let $H=\R^2$ and the dynamical system $S(t)$ is determined by the phase portrait below

\vspace{2mm}

{
\hspace{35mm}
\begin{tikzpicture}[scale=0.35]
\fill[outer color=red!5] (0,1) circle (2);
\draw[thick, ->] (0,-1) arc (270:360:0.5);
\draw[thick, ->] (0.5,-0.5) arc (0:180:0.5);
\draw[thick, ->] (-0.5,-0.5) arc (180:270:0.5);
\draw[thick, ->] (1,0) arc (0:120:1);
\draw[thick, ->] (-0.5,0.8660254) arc (120:360:1);
\draw[thick, ->] (0,-1) arc (270:360:1.5);
\draw[thick, ->] (1.5,0.5) arc (0:180:1.5);
\draw[thick, ->] (-1.5,0.5) arc (180:270:1.5);
\draw[color=red!70, thick, ->] (0,-1) arc (270:360:2);
\draw[color=red!70, thick, ->] (2,1) arc (0:90:2);
\draw[color=red!70, thick, ->] (0,3) arc (90:180:2);
\draw[color=red!70, thick, ->] (-2,1) arc (180:270:2);
\draw[thick, ->](0,3.8) arc (90:230:2.6);
\draw[thick, ->](-1.6712478,-0.791715544) to [out=330, in=195] (0,-1);
\draw[thick, ->](2.7,1.1) arc (0:90:2.7);
\draw[thick](2.7,1.1) arc (360:300:2.7);
\draw[thick, ->] (0.5,-3) to [out=85, in=195] (1.35,-1.23826858);
\draw[thick, ->] (0.5,-4)--(0.5,-3);
\draw[thick, ->] (0,-4)--(0,-2);
\draw[thick, ->] (0,-2)--(0,-1);
\draw[thick, ->](0,5.8) arc (90:180:4.7);
\draw[thick, ->](-4.7,1.1) arc (180:230:4.7);
\draw[thick, ->](-3.021101767,-2.500408868) to [out=330, in=260] (0,-1);
\draw[thick](0,5.8) arc (90:30:4.9);
\draw[thick, ->](4.24352446,-1.55) arc (330:390:4.9);
\draw[thick, ->] (2.5,-4) to [out=85, in=230] (4.24352446,-1.55);
\draw[thick, ->] (0.5,-4)--(0.5,-3);
\draw[thick, ->](-5,-4) to [out=350, in=265] (-0.05,-1.4);
\draw[thick,](0,-1) -- (-0.05,-1.4);
\end{tikzpicture}
}

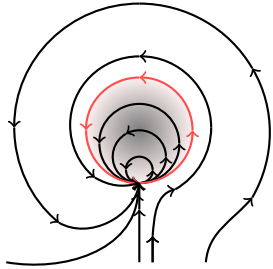
    {\centerline{Figure 1.}}
\vspace{3mm}
\par
Actually this phase portrait consists of a saddle-node at the origin $(0,0)$ glued with the disc $\{(x,y)\in\R^2\,:\  x^2+(y-1)^2\le 1\}$ filled by homoclinic orbits to the origin. The key feature of this DS is that the $\omega$-limit set of any single trajectory coincides with the origin (and, therefore, any continuous functional is automatically asymptotically determining), but the global attractor
$$
\Cal A=\{(x,y)\in\R^2,\ \ x^2+(y-1)^2\le1\}
$$
is not trivial. Obviously, the functional $F(x,y)=x^2+(y-1)^2-1$ is not separating on the attractor since it does not distinguish the origin and the homoclinic orbit at the boundary of the attractor.
\end{example}
Next example shows that the class of linear functionals may be insufficient for finding good determining functionals.
\begin{example}\label{Ex4.linbad} Let $H=\R^4$ and let the dynamical system $S(t)$ be determined by the following system of linear ODEs:
\begin{equation}\label{4.linear}
\dot x=y,\ \ \dot y=-x,\ \ \dot z=u,\ \ \dot u=-z.
\end{equation}
Of course, this DS is not dissipative, but we may correct the right-hand sides outside of a large ball making it dissipative. A general solution of this system has the form:
\begin{equation*}
x=A\sin(t+\phi_1),\ y=A\cos(t+\phi_1),\ \ z=B\sin(t+\phi_2),\ \ u=B\cos(t+\phi_2),
\end{equation*}
where the amplitudes $A,B$ and phases $\phi_1$ and $\phi_2$ are arbitrary. We see that this system has only one equilibrium (this property can be easily preserved under the correction of the right-hand sides for making the system dissipative), therefore, we know from a general theory that the determining dimension of this system is one. In {  a} fact, it is not difficult to see that the desired functional can be found in the class of quadratic functionals even for the initial system \eqref{4.linear}. We claim that such a functional cannot be linear. Indeed, let
$$
F=\alpha x+\beta y+\gamma z+\delta u
$$
be such a functional where not all of $\alpha,\beta,\gamma,\delta$ are equal to zero. We claim that there are non-zero $A,B,\phi_1,\phi_2$ such that
\begin{equation}\label{4.trig}
\alpha\cos(t-\phi_1)+\beta\sin(t-\phi_1)+\gamma\cos(t-\phi_2)+\delta\sin(t-\phi_2)\equiv0.
\end{equation}
Indeed, after the elementary transformations, \eqref{4.trig} is equivalent to
$$
AA_{\alpha,\beta}\cos(t-\phi_1-\phi_{\alpha,\beta})+BA_{\gamma,\delta}\cos(t-\phi_2-\phi_{\gamma,\delta})\equiv0
$$
for the appropriate amplitudes $A_{\alpha,\beta}$ and $A_{\gamma,\delta}$ and phases $\phi_{\alpha,\beta}$ and $\phi_{\gamma,\delta}$ depending only on $\alpha,\beta,\gamma,\delta$. In turn, this identity will be satisfied if
$$
\sin(\phi_1-\phi_2+\phi_{\alpha,\beta}-\phi_{\gamma,\delta})=0
$$
 and
$$
AA_{\alpha,\beta}+BA_{\gamma,\delta}\cos(\sin(\phi_1-\phi_2+\phi_{\alpha,\beta}-\phi_{\gamma,\delta})=0.
$$
The first equation can be satisfied by the choice of phases $\phi_1$ and $\phi_2$. Then the second one becomes a linear equation on the amplitudes $A$ and $B$ which always has a non-trivial solution. Thus, $F$ cannot be a determining functional.
\end{example}
The next two examples are related to the case of one spatial dimension.
\begin{example}\label{Ex4.dir} Let us consider the following 1D semilinear heat equation
\begin{equation}\label{4.heat}
\Dt u=\nu\partial_x^2u-f(u)+g,\ \ x\in[0,\pi], \ \ \nu>0
\end{equation}
endowed with Dirichlet boundary conditions
$$
u\big|_{x=0}=u\big|_{x=L}=0.
$$
Assume also that $f\in C^1(\R,\R)$ satisfies some dissipativity conditions, say $f(u)u\ge-C$. Then, equation \eqref{4.heat} generates a dissipative DS in the phase space $H=L^2(0,L)$ and this DS possesses a global attractor $\Cal A$ which is bounded at least in $C^2([0,L])$, see e.g. \cite{BV92},
\par
Moreover, the equilibria $\Cal R$ of this problem satisfies the second order ODE
$$
\nu u''(x)-f(u(x))+g=0,\ \ u(0)=u(L)=0.
$$
Thus, the map $u\to u'(0)$ gives a homeomorphic (and even smooth) embedding of $\Cal R$ to $\R$, so $\dim_{emb}(\Cal R)=1$. Thus, we expect that $\dim_{det}(S(t),H)=1$ (or at most $2$ according to Corollary \ref{Cor3.det-R}). The possible explicit form of the determining functional is well-known here: $F(u):= u\big|_{x=x_0}$, where $x_0>0$ is small enough, see \cite{kukavica}. Indeed, let $u_1(t),u_2(t)\in\Cal A$ be two complete trajectories of \eqref{4.heat} belonging to the attractor such that $u_1(t,x_0)\equiv u_2(t,x_0)$. Then the function $v(t)=u_1(t)-u_2(t)$ solves
\begin{equation}\label{4.dif}
\Dt v=\nu\partial_x^2v-l(t)v,\  v\big|_{x=0}=v\big|_{x=x_0}=0,
\end{equation}
where $l(t):=\int_0^1f'(su_1(t)+(1-s)u_2(t))\,ds$. Since the attractor $\Cal A$ is bounded in $C$, we know that $|l(t)|_{C[0,L]}\le L$ is also globally bounded. Multiplying equation \eqref{4.dif} by $v$, integrating in $x\in[0,x_0]$ and using the fact that the first eigenvalue for $-\partial_x^2$ with Dirichlet boundary conditions is $(\frac\pi{x_0})^2$, we get
$$
\frac12\frac d{dt}\|v(t)\|^2_{L^2}+\nu\(\frac\pi{x_0}\)^2\|v(t)\|^2_{L^2}-L\|v(t)\|^2_{L^2}\le0
$$
Fixing $x_0>0$ to be small enough that $\nu\(\frac\pi{x_0}\)^2>L$, applying the Gronwall inequality and using that $\|v(t)\|_{L^2}$ remains bounded as $t\to-\infty$, we conclude that $v(t)\equiv0$ for all $t\in\R$ and $x\in[0,x_0]$. Thus, the trajectories $u_1(t,x)$ and $u_2(t,x)$ coincide for all $t$ and all $x\in[0,x_0]$. Using now the arguments related to logarithmic convexity or Carleman type estimates which work for much more general class of equations, see \cite{RL}), we conclude that the trajectories $u_1$ and $u_2$ coincide. Thus, $F$ is separating on the attractor and therefore is asymptotically determining.
Being pedantic, we need to note that the functional $F(u)=u\big|_{x=x_0}$ is not defined on the phase space $H$, but on its proper subspace $C[0,L]$ (this is a typical situation for determining nodes, see \cite{OT08}). However, we have an instantaneous $H\to C[0,L]$ smoothing property, so if we start from $u_0\in H$, the value $F(u(t))$ will be defined for all $t>0$, so we just ignore this small inconsistency.
\end{example}
\begin{example}\label{Ex4.per} Let us consider the same equation \eqref{4.heat} on $[0,L]$, but endowed with {\it periodic} boundary conditions. In this case we do not have the condition $v(0)=0$, so the set $\Cal R$ of equilibria is naturally embedded in $\R^2$, not in $\R^1$, by the map $u\to (u\big|_{x=0},u'\big|_{x=0})$. Thus, we cannot expect that the determining dimension is one. Moreover, at least in the case when $g=const$, equation \eqref{4.heat} possess a spatial shift as a symmetry and, therefore, any nontrivial equilibrium generates the whole circle of equilibria. Since a circle cannot be homeomorphically embedded in $\R^1$, the determining dimension must be at least $2$. We claim that it is indeed $2$ and the determining functionals can be taken in the form:
\begin{equation}\label{4.2-det}
F_1(u):=u\big|_{x=0},\ \ F_2(u)=u\big|_{x=x_0}.
\end{equation}
Indeed, arguing exactly as in the previous example, we see that the system $\Cal F=\{F_1,F_2\}$ is asymptotically determining if $x_0>0$ is small enough.
\end{example}
The next natural example shows that the determining dimension may be finite and small even if the corresponding global attractor is infinite-dimensional.
\begin{example}\label{Ex4.inf} Let us consider the semilinear heat equation \eqref{4.heat} on the whole line $x\in\R$. The natural phase space for this problem is the so-called {\it uniformly-local} space
\begin{equation}\label{4.ul}
H=L^2_b(\R):=\{u\in L^2_{loc}(\R),\ \|u\|_{L^2_b}:=\sup_{x\in\R}\|u\|_{L^2(x,x+1)}<\infty\}.
\end{equation}
It is known that, under natural dissipativity assumption on $f\in C^1(\R)$ (e.g., $f(u)u\ge-C+\alpha u^2$ with $\alpha>0$), this equation generates a dissipative DS $S(t)$ in $H$ for every $g\in L^2_b(\R)$. Moreover, this DS possesses the so called {\it locally compact} global attractor $\Cal A$ that  is a bounded in $H$ and compact in $L^2_{loc}(\R)$ strictly invariant set which attracts bounded in $H$ sets in the topology of $L^2_{loc}(\R)$, see \cite{MZ08} and references therein. Note that in contrast to the case of bounded domains the compactness and attraction property in $H$ fail in general in the case of unbounded domains. It is also known that, at least in the case where equation \eqref{4.heat} possesses a spatially homogeneous exponentially unstable equilibrium (e.g., in the case where $g=0$ and $f(u)=u^3-u$), the box-counting dimension of $\Cal A$ is infinite (actually it contains submanifolds of any finite dimension), see \cite{MZ08}.
\par
Nevertheless, the system of linear functionals \eqref{4.2-det} remains determining for this equations by {\it exactly} the same reasons as in Examples \ref{Ex4.dir} and \ref{Ex4.per}. Thus,
$$
\dim_{det}(S(t),H)=2.
$$
One determining functional does not exist in general by the reasons explained in Example \ref{Ex4.per}.
\end{example}
Next example shows that in some extremely degenerate cases, the dimension of the attractor may coincide with the embedding dimension of the set of equilibria and with the number $N$ of determining Fourier modes from Proposition \ref{Prop2.modes}.
\begin{example}\label{Ex4.deg} Let the assumptions of Proposition \ref{Prop2.modes} hold and let for simplicity $\alpha=0$ (i.e., $f$ is globally Lipschitz as a map from $H$ to $H$ with the Lipschitz constant $L$).  Let us define  $N$ as a minimal natural number for which $L<\lambda_{N+1}$ be satisfied. Let us fix a smooth function $\phi:\R\to\R$, $\phi\in C_0^\infty(\R)$ in such a way that $|\phi'(x)|\le1$ for all $x$ and $\phi(x)=x$ for all $x\in[-1,1]$. Then, finally, take
\begin{equation}\label{4.deg}
f(u)=\sum_{n=1}^N\lambda_n\phi((u,e_n))e_n,\ \ g\equiv0.
\end{equation}
It is clear that $f$ is bounded and globally Lipschitz with Lipschitz constant $\lambda_N<L$, so equation \eqref{2.abs} possesses a global attractor $\Cal A$ whose fractal dimension is exactly $N$. Indeed, by the construction of $f$, $Q_Nf(u)\equiv0$ and, therefore, $\Cal A\subset P_NH=\R^n$. On the other hand, by the construction   of $f$, the set $\Cal R$ of equilibria contains a cube $[-1,1]^N\subset P_NH=\R^N$ and, therefore, $\dim_{emb}(\Cal R)=N$. On the other hand, by Proposition \ref{Prop2.modes}, the first $N$ Fourier modes are asymptotically determining and, therefore,
$$
\dim_{det}(S(t),H)=N.
$$
\end{example}
We now give an example of a non-dissipative and even {\it conservative}  system with determining dimension one.
\begin{example}\label{Ex4.wave} Let us consider the following 1D wave equation:
\begin{equation}\label{4.wave}
\partial^2_tu=\partial_x^2u,\ \ x\in(0,\pi),\ \ u\big|_{x=0}=u\big|_{x=\pi}=0,\ \ \xi_u\big|_{t=0}=\xi_0,
\end{equation}
where $\xi_u(t):=\{u(t),\Dt u(t)\}$. It is well-known that problem \eqref{4.wave} is well-posed in the energy phase space $E:=H^1_0(0,\pi)\times L^2(0,\pi)$ and the energy identity holds
\begin{equation}
\|\Dt u(t)\|_{L^2}^2+\|\partial_x u(t)\|^2_{L^2}=const.
\end{equation}
Moreover, the solution $u(t)$ can be found explicitly in terms of $\sin$-Fourier series:
\begin{equation}\label{4.sol}
u(t)=\sum_{n=1}^\infty (A_n\cos(nt)+B_n\sin(nt))\sin (nx),
\end{equation}
where $A_n=\frac2{\pi}(u(0),\sin (nx)$, $B_n=\frac2{n\pi}(u'(0),\sin(nx))$. Crucial for us is that the function \eqref{4.sol} is an {\it almost-periodic} function of time with values in $H^1_0$. Let us consider now a linear functional on $H=L^2(-\pi,\pi)$, i.e,
\begin{equation}
Fu=(l(x),u)=\sum_{n=1}^\infty l_nu_n,\ \
\end{equation}
where $l_n$ and $u_n$ are the Fourier coefficients of $l\in H$ and $u$ respectively. Then
\begin{equation}
Fu(t)=\sum_{n=1}^\infty l_nA_n\cos(nt)+l_nB_n\sin(nt)
\end{equation}
is a scalar almost periodic function. Since the Fourier coefficients of an almost-periodic function are uniquely determined by the function, we have
$$
Fu(t)\equiv0\ \ \text{ if and only if }\ \ l_nA_n=l_nB_n=0,
$$
see \cite{LZ} for details. Thus, if we take a generic function $l$ (for which $l_n\ne0$ for all $n\in\Bbb N$, $Fu(t)\equiv0$ will imply that $A_n=B_n=0$ and, therefore, $u(t)\equiv0$. Thus, $F$ is separating on the set of complete trajectories. It remains to note that, since any trajectory of \eqref{4.wave} is almost-periodic in $E$, the $\omega$-limit set of any trajectory exists and is compact in $E$.  Then, by Proposition \ref{Prop1.det} and Remark \ref{Rem1.strange}, $F$ is also asymptotically determining, so the determining dimension of this system is one.
\end{example}
We conclude this section by a bit counterintuitive example which shows that multiple instability can be removed by the feedback control based on one mode. Such examples are well-known  in the control theory (in a fact, they are particular cases of the famous Kalman rank theorem, see \cite{Kal,Kal1}), but have been somehow overseen in the theory of determining functionals. This example also gives an alternative explicit form for the extra terms in Proposition \ref{Prop2.assym} and \ref{Prop3.assym}.

\begin{example}\label{Ex4.feed} Let us consider the following linear problem
\begin{equation}\label{4.linheat}
\Dt u=\nu\partial_x^2 u-a(x)u,\ \ x\in(0,L),\ \ \ u\big|_{x=0}=u\big|_{x=L}=0,
\end{equation}
where $a(x)$ is a given function (belonging to $C[0,L]$ for simplicity). We would like to stabilize zero equilibrium of this equation by the rank-one linear feedback control given by the term $(l(x),u)w(x)$ for the properly chosen functions $l,w\in H:=L^2(0,L)$. In other words, we need to find $l$ and $w$ such that the equation
\begin{equation}\label{4.feed}
\Dt v=\nu\partial_x^2v-a(x)v+(l,v)w, \ \ \ v\big|_{x=0}=v\big|_{x=L}=0
\end{equation}
become exponentially stable. Let $A=\nu\partial_x^2-a(x)$ be a self-adjoint operator in $H=L^2(0,L)$ endowed by  homogeneous Dirichlet boundary conditions and let $\mu_n$ and $\xi_n$ be the corresponding eigenvalues and eigenvectors enumerated in the non-increasing order. Then, since $\mu_n\to-\infty$ as $n\to\infty$, we have only finite unstable eigenmodes and let $N$ be the number of such modes. Thus, the higher modes are already stable, so we need not to stabilize them and we may assume that $l,w\in P_NH$. This reduces our problem to the finite dimensional problem in $\R^N=P_NH$:
\begin{equation}\label{4.fin}
\Dt v=A_Nv+(l,v)w,\ \ A_N=\operatorname{diag}\{\mu_1,\cdots,\mu_N\},\ \ l,v,w\in\R^N.
\end{equation}
Crucial for us is that the eigenvalues of $A$ are all {\it simple} (here we have essentially used that our problem is scalar and has only one spatial dimension), therefore, the matrix $A_N$ possesses a cyclic vector $e=e_1$. Let us consider the corresponding  cyclic base $\{e_1,\cdots,e_N\}$ in $\R^N$:
$$
Ae_1=e_2,\cdots, Ae_{N-1}=e_N,\ \ Ae_N=\alpha_1e_1+\cdots\alpha_Ne_N.
$$
In this base, the matrix $A_N$ reads
\begin{equation}
A_N=\(\begin{matrix}0&1&0&0&\cdots&0\\
                                   0&0&1&0&\cdots&0\\
\vdots&\vdots&\vdots&\vdots&\cdots&\vdots\\
                                   0&0&0&0&\cdots&1\\
                                   \alpha_1&\alpha_2&\alpha_3&\alpha_4&\cdots&\alpha_N
\end{matrix}\)
\end{equation}
Note also that the coefficients $\alpha_1,\cdots,\alpha_N$ are nothing more than the coefficients of the characteristic polynomial of the matrix $A_N$. Thus, controlling these coefficients, we control the spectrum of the matrix $A_N$. Let us now take $w=e_N$. Then, the matrix which corresponds to the rank-one operator $(l,u)e_N$ reads
\begin{equation}
\(\begin{matrix}0&0&0&0&\cdots&0\\
                                   0&0&0&0&\cdots&0\\
\vdots&\vdots&\vdots&\vdots&\cdots&\vdots\\
                                   0&0&0&0&\cdots&0\\
                                   l_1&l_2&l_3&l_4&\cdots&l_N
\end{matrix}\).
\end{equation}
Thus, choosing the coefficients $l_1,\cdots,l_N$ in a proper way, we can change the characteristic polynomial of $A_N+(l,\cdot)e_N$ arbitrarily. In particular, we can make all its roots negative and this gives us the desired feedback control.
\end{example}

\section{Open problems}\label{s5}
In this concluding section we briefly discuss some interesting open problems related to the theory of determining functionals. We start with a natural question related to Navier-Stokes equations.
\begin{problem}\label{Pr5.pressure} It looks natural and interesting to study the non-linear determining functionals satisfying some extra restrictions. For instance, let us consider the Navier-Stokes problem:
\begin{equation}\label{5.NS}
\Dt u+(u,\Nx)u+\Nx p-\Dx u=g,\ \ \divv u=0
\end{equation}
in a bounded domain $\Omega$. Here $u$ is an unknown velocity field and $p$ is an unknown pressure. The following question is of big both theoretical and applied interests:
\par
{\it Is it possible to find determining functionals depending  on pressure $p$ only?}
\par
Note that the pressure can be expressed through velocity via
$$
\Dx p=-\sum_{i,j}\partial_{x_i}\partial_{x_j}(u_iu_j)+\operatorname{div} g
$$
and, therefore, even if the desired functionals will be linear in pressure, they can be expressed as {\it non-linear} (quadratic) ones in terms of the velocity vector field. This problem is one of possible extra motivations to study the non-linear determining functionals.
\end{problem}
Next problem is related to the actual smoothness of the reduction to DDEs.
\begin{problem}\label{Pr5.smo} We recall that the non-linearity $\bar\Theta$ in the constructed delayed ODE \eqref{3.Z-ode} is a priori continuous only although we expect that it general it is at least H\"older continuous (to verify this one should use H\"older Mane projection theorem together with the appropriate version of Takens theorem). The situation with further smoothness of $\bar\Theta$ is more delicate. Indeed, the original finite-dimensional version of Takens  theorem, see \cite{Tak} deals with smooth maps and diffeomorphisms, however, the infinite-dimensional version is usually verified using the H\"older Mane projection theorem as an intermediate step and this leads to a drastic loss of smoothness making everything H\"older continuous only, see \cite{SYC91} or \cite{R11}. It is not clear how essential this loss of smoothness is and whether or not the problem can be overcome using more sophisticated techniques and allowing the infinite delay.
\par
It also worth mentioning that the nonlinearity $\Phi_0$ in the particular case  of the reduction related to Fourier modes is smooth (its smoothness is restricted by the smoothness of the map $\Phi_0$ only) which gives a hope that the smoothness problem can be overcome in a general case as well. On the other hand, as known, for the "true" finite-dimensional reduction related to Man\'e projection theorem, this loss of smoothness is crucial and, to the best of our knowledge, can be overcome only in the case where the considered DS possesses an IM, see \cite{R11, Z14} and references therein.
\end{problem}
Our final open problem is related to application to data assimilation, namely, with  explicit constructions of recovering equations \eqref{3.as} in the cases where the explicit construction of determining functionals is available.
\begin{problem}\label{Pr5.Edr} We know from Example \ref{Ex4.dir} that reaction-diffusion equation \eqref{4.heat} possesses a single determining functional $F(u):=u\big|_{x=x_0}$, where $x_0>0$ is small enough. Therefore, according to Proposition \ref{Prop3.assym}, any trajectory $u(t)$ on the attractor can be recovered by its values at point $x=x_0$ by solving equation \eqref{3.as}. However, the corresponding function $\Theta_N$ involved into this equation is rather complicated and requires us to compute the attractor $\Cal A$ as well as to verify that $F$ satisfies  the Takens theorem before we will be able to use equations \eqref{3.as}. On the other hand, as shown in \cite{AT14}, one can use a very simple explicit form of function $\Theta_N$ (similar to what is used in \eqref{2.as}) if we take sufficiently many nodes $0<x_1<x_2<\cdots<x_N<L$ and the corresponding functionals. Thus, a natural question arises here:
\par
{\it Is it possible to find a simpler form of equations \eqref{3.as} which do not require to compute the attractor $\Cal A$?}
\par
Example \ref{Ex4.feed} gives a partial answer to this question for the case of {\it linear} 1D parabolic equation. However, it is not clear how to extend it to the non-linear case of equation \eqref{4.heat} and it would be also nice to use the observation functional in the form $F(u)=u\big|_{x=x_0}$.
\par
Revising Example \ref{Ex4.dir}, we see that this problem is closely related to the regularization of the ill posed (overdetermined) boundary value problem:
$$
\Dt w=\nu\partial_x^2w-l(t)w,\ \ w\big|_{x=x_0}=\partial_x w\big|_{x=x_0}=w\big|_{x=L}=0.
$$
We expect that, regularizing this problem properly and using the Carleman type estimates
on the interval $(x_0,L)$ (see e.g. \cite{RL}), one should be able to construct the approximating solution $v(t,x)$ in such a way that $v(t,x)\to u(t,x)$ as $t\to\infty$ exponentially fast on the whole interval $x\in[0,L]$. We will return to this topic  somewhere else.
\end{problem}


\begin{thebibliography}{9}
\bibitem{AT14}
A. Azouani  and E. Titi,  {\it Feedback control of nonlinear dissipative systems by finite determining
parameters—a reaction–diffusion paradigm}, Evol. Equ. Control Theory 3 (2014) 579--594.
\bibitem{AT13}
A. Azouani, E. Olson  and E. Titi, {\it Continuous data assimilation using general interpolant observables}, J.
Nonlinear Sci. 24 (2013) 1--27.
\bibitem{BV92}
A. Babin and M. Vishik, {\it Attractors of evolution equations,} Studies in Mathematics and its
Applications, 25. North-Holland Publishing Co., Amsterdam, 1992.
\bibitem{3}
A. Ben-Artzi, A. Eden, C. Foias, and B. Nicolaenko,
{\it H\"older continuity for the inverse of
Mane’s projection.}
J. Math. Anal. Appl., vol.178, (1993) 22--29.
\bibitem{CV02}
V. Chepyzhov and M. Vishik, {\it Attractors for equations of mathematical physics,} American Mathematical Society Colloquium Publications, 49. American Mathematical Society,
Providence, RI, 2002.
\bibitem{Chu}
I. D. Chueshov, {\it Theory of functionals that uniquely determine the asymptotic
dynamics of infinite-dimensional dissipative systems}, Uspekhi Mat. Nauk, 1998,
Volume 53, Issue 4(322), 77–124.
\bibitem{Chu1} I. D. Chueshov, {\it Dynamics of quasi-stable dissipative systems.} Universitext, Springer,
Cham, 2015.
\bibitem{Cock1} B. Cockburn, D.A. Jones and E.S. Titi, {\it Degr\'es de libert\'e d\'eterminants pour \'equations
non lin\'eaires dissipatives}, C.R. Acad. Sci.-Paris, S\'er. I 321 (1995), 563--568 .
\bibitem{Cock2}
B. Cockburn, D.A. Jones and E.S. Titi, {\it Estimating the number of asymptotic degrees
of freedom for nonlinear dissipative systems}, Math. Comput. 97 (1997), 1073--1087.
\bibitem
{CFNT89}
 P. Constantin, C. Foias, B. Nicolaenko, and R. Temam,
 \emph{Inertial Manifolds for Dissipative Partial Differential Equations (Applied Mathematical Sciences,
  no. 70)}, Springer-Verlag, New York, 1989.
\bibitem{EKZ13}
A. Eden, V. Kalanarov and S. Zelik, {\it Counterexamples to the regularity of Mane projections and global
attractors,} Russian Math Surveys, vol. 68, no. 2, (2013) 199--226.
\bibitem{FMRT} C. Foias, O.P. Manley, R. Rosa, R. Temam, {\it Navier-Stokes Equations and Turbulence},
Cambridge University Press, 2001.
\bibitem{FP67} C. Foias and G. Prodi, {\it Sur le comportement global des solutions non-stationnaires des equations de Navier–Stokes en dimension 2}, Rend. Sem. Mat. Univ. Padova 39 (1967) 1–34.
\bibitem{FT} C. Foias, R. Temam, {\it Determination of the solutions of the Navier-Stokes equations
by a set of nodal values,} Math. Comp. 43, no. 167 (1984)  117--133.
\bibitem{FTT}
C. Foias, E.S. Titi, {\it Determining nodes, finite difference schemes and inertial mani-
folds,} Nonlinearity 4, no. 1 (1991) 135--153.
\bibitem
{FST88}
 C. Foias, G. Sell, and  R. Temam,
 \emph{Inertial manifolds for nonlinear evolutionary equations},
 J. Differential Equations, vol. 73, no. 2, (1988) 309--353.
\bibitem{ha}
{\au J. Hale},
{\bk Asymptotic Behaviour of Dissipative Systems},
\eds{Math. Surveys and Mon.}{AMS Providence, RI}{1987}
\bibitem{HR}
J. Hale and G. Raugel, {\it Regularity, determining modes and Galerkin methods}, Journal de Math\'ematiques Pures et Appliqu\'ees, vol. 82, no. 6 (2003) 1075--1136.
\bibitem{hen} D. Henry, {\it Geometric theory of semilinear parabolic equations.} Lecture Notes in Mathematics,
840. Springer-Verlag, Berlin–New York, 1981.
\bibitem{hunt}
B. Hunt and V. Kaloshin,
{\it Regularity of embeddings of infinite-dimensional fractal sets
into finite-dimensional spaces.}
Nonlinearity,  vol. 12, (1999) 1263--1275.
\bibitem{JT} D.A. Jones, E.S. Titi, {\it Upper bounds on the number of determining modes, nodes, and
volume elements for the Navier-Stokes equations,} Indiana Univ. Math. J. 42, no. 3 (1993), 875--887.
\bibitem{Kal}
R. Kalman,  {\it Contributions to the theory of optimal control}, Bol Soc Mat Mex 5 (1960)
102--119.
\bibitem{Kal1} R. Kalman,  Y. Ho, and K. Narendra, {\it Controllability of linear dynamical systems},
Contr Diff Eqs 1 (1963) 189--213.
\bibitem{KZ18}
A. Kostianko and S. Zelik,
 {\it Inertial manifolds for 1D reaction-diffusion-advection systems. Part II:
Periodic boundary conditions,} Commun. Pure Appl. Anal., vol. 17, no. 1, (2018)  265--317.
\bibitem{KZ17}
 A. Kostianko and S. Zelik,
 {\it Inertial manifolds for 1D reaction-diffusion-advection systems. Part I:
Dirichlet and Neumann boundary conditions,} Commun. Pure Appl. Anal., vol. 16, no. 6, (2017) 2357–
2376.
\bibitem{kukavica}
I. Kukavica, {\it On the number of determining nodes for the Ginzburg-Landau equation}, Nonlineaity 5 (1992) 997--1006.
\bibitem{kuksin}
S. Kuksin and A. Shirikyan, {\it Mathematics of Two-Dimensional Turbulence},
Cambridge Tracts in Mathematics 194, Cambridge University Press, 2012.
\bibitem{Lad72} O. Ladyzhenskaya,  {\it On a dynamical system generated by Navier–Stokes equations}, Zap. Nauchn. Sem. LOMI, Volume 27 (1972)	 91--115.
\bibitem{LZ}
B. Levitan and V. Zhikov, {\it Almost-periodic functions and differential equations}, Cambridge University Press, 1982.
\bibitem
{M-PS88}
J. Mallet-Paret and G. Sell,
\emph{Inertial manifolds for reaction diffusion equations in higher space dimensions},
J. Am. Math. Soc., vol. 1, no. 4, (1988) 805--866.
\bibitem{MPSS93}
J. Mallet-Paret, G. Sell, and Z. Shao, {\it Obstructions to the existence of normally hyperbolic
inertial manifolds,} Indiana Univ. Math. J., 42, no. 3, (1993) 1027–1055.
\bibitem
{M91}
M. Miklavcic,
\emph{A sharp condition for existence of an inertial manifold},
J. Dyn. Differ. Equations, vol. 3, no. 3, (1991) 437--456.
\bibitem
{MZ08}
 A. Miranville and S. Zelik, \emph{Attractors for dissipative partial differential
  equations in bounded and unbounded domains},
   in: Handbook of Differential Equations: Evolutionary Equations,
    vol. IV, Elsevier/North-Holland, Amsterdam, 2008.
\bibitem{OT03}
E. Olson and E. Titi, {\it Determining modes for continuous data assimilation in 2D turbulence}, J. Stat. Phys.
113 (2003) 799--840.
\bibitem{OT08} E. Olson and E. Titi, {\it Determining modes and Grashof number in 2D turbulence}, Theor. Comput. Fluid
Dyn. 22(2008) 327--339.
\bibitem
{R01}
J. Robinson,
\emph{Infinite-dimensional Dynamical Systems},
Cambridge University Press, 2001.
\bibitem{R11}
J. Robinson, {\it Dimensions, embeddings, and attractors,} Cambridge University Press,
Cambridge, 2011.
\bibitem
{R94}
A.  Romanov, \emph{Sharp estimates for the dimension of inertial manifolds
for nonlinear parabolic equations},
Izv. Math. vol. 43, no. 1, (1994) 31--47.
\bibitem{Rom00}
A. Romanov, {\it Three counterexamples in the theory of inertial manifolds,} Math. Notes, vol. 68,
no. 3–4, (2000) 378--385.
\bibitem{RL}
J. Rousseau and G. Lebeau,
{\it On Carleman Estimates for Elliptic and Parabolic Operators.
Applications to Unique Continuation and Control
of Parabolic Equations}, ESAIM: COCV 18 (2012) 712--747.
\bibitem{SYC91} T. Sauer, J. Yorke, and M. Casdagli,  {\it Embedology,} J. Stat. Phys. 71(1991) 529–547.
\bibitem
{SY02}
G. Sell and Y. You, \emph{Dynamics of evolutionary equations}, Springer, New York, 2002.
\bibitem{stein}
E. Stein, {\it Singular Integrals and Differentiability Properties of Functions}, Princeton Univ. Press, Princeton, 1970.
\bibitem{Tak} F. Takens, {\it Detecting strange attractors in turbulence}, Lecture Notes in Mathematics No. 898, (1981) 366--381.
\bibitem
{T97}
 R. Temam,
 \emph{Infinite-Dimensional Dynamical systems in Mechanics and Physics},
 second edition, Applied Mathematical Sciences, vol 68, Springer-Verlag, New York, 1997.
\bibitem
{Z14}
S. Zelik,
\emph{Inertial manifolds and finite-dimensional reduction for dissipative PDEs},
Proc. Royal Soc. Edinburgh 144, vol. 6, (2014) 1245--1327.
\end{thebibliography}
\end{document}